\documentclass{amsart}
\usepackage{amsmath, latexsym, amsfonts, amssymb, amsthm, amscd, color}
\newtheorem {definition}{Definition}

\newtheorem {lemma}[definition]{Lemma}

\newtheorem {theorem}[definition]{Theorem}
\newtheorem {rem}[definition]{Remark}

\newcommand{\E}{\mathbb{E}}
\newcommand{\N}{\mathbb{N}}
\renewcommand{\P}{\mathbb{P}}

\newcommand{\R}{\mathbb{R}}
\newcommand{\eps}{\varepsilon}

\author{Leif D\"oring}
\address{Leif D\"oring, Departement Mathematik, ETH Z\"urich, R\"amistrasse 10, 8092 Z\"urich, Switzerland}

\author{Andreas E. Kyprianou}
\address{
Andreas E. Kyprianou, 
Department of Mathematical Sciences,
University of Bath,
Claverton Down,
Bath BA2 7AY,
United Kingdom
}

\title{Perpetual Integrals for L\'evy Processes}
\begin{document}
\maketitle
\begin{abstract}
	Given a L\'evy process $\xi$ we ask for necessary and sufficient conditions for almost sure finiteness of the perpetual integral $\int_0^\infty f(\xi_s)ds$, where $f$ is a positive locally integrable function. If $\mu=\E[\xi_1]\in (0,\infty)$ and $\xi$ has local times we prove the 0-1 law $$\P\Big(\int_0^\infty f(\xi_s)\,ds<\infty\Big)\in \{0,1\}$$ with the exact characterization
\begin{align*}
	\P\Big(\int_0^\infty f(\xi_s)\,ds<\infty\Big)=0\qquad \Longleftrightarrow \qquad \int^\infty f(x)\,dx=\infty.
\end{align*}
The proof uses spatially stationary L\'evy processes, local time calculations, Jeulin's lemma and the Hewitt-Savage 0-1 law.\\\bigskip

\end{abstract}

	The study of perpetual integrals $\int f(X_s)ds$ with finite or infinite horizon for diffusion processes $X$ has a long history partially because of their use in the analysis of stochastic differential equations and insurance, financial mathematics as the present value of a continuous stream of perpetuities.
	
\smallskip
The main result of the present article is a characterization of finiteness for perpetual integrals of L\'evy processes:
\begin{theorem}\label{theorem}
	Suppose that $\xi$ is a L\'evy process that has strictly positive mean $\mu<\infty$, local times and is not a compound Poisson process. If $f$ is a measurable locally integrable positive function, then the following 0-1 law holds:
	\begin{align}\label{statementa}
	\P\Big(\int_0^\infty f(\xi_s)\,ds<\infty\Big)=1\qquad &\Longleftrightarrow \qquad \int^\infty f(x)\,dx<\infty\tag{T1}
\end{align}
	and
\begin{align}\label{statementb}
	\P\Big(\int_0^\infty f(\xi_s)\,ds<\infty\Big)=0\qquad &\Longleftrightarrow \qquad \int^\infty f(x)\,dx=\infty.\tag{T2}
\end{align}
\end{theorem}
 Let us briefly compare the theorem with the existing literature:
 
 \smallskip
  (i) If $\xi$ is a Brownian motion with positive drift, then results were obtained through the Ray-Knight theorem, Jeulin's lemma and Khashminkii's lemma by Salminen/Yor \cite{Yor}, \cite{Yor2}.
  
  \smallskip
  (ii) For spectrally negative $\xi$, i.e. $\xi$ only jumps downwards, the equivalence was obtained in Khoshnevisan/Salminen/Yor \cite{Yor2}, see also Example 3.9 of Schilling/Voncracek \cite{Schilling}. The spectrally negative case also turns out to be easier in our proof.
  
   \smallskip
(iii) If $f$ is (ultimately) decreasing, results for general L\'evy processes have been proved in Erickson/Maller \cite{Maller}. In this case the result stated in Theorem \ref{theorem} follows easily from the law of large numbers by estimating $2\mu t>\xi_t>\frac 1 2 \mu t$ for $t$ big enough. The very same argument also shows that for $\mu=+\infty$ the integral test (T) fails in general. For a result in this case we again refer to \cite{Maller}.
\begin{rem}
	It is not clear whether or not the assumption $\xi$ having local time plays a role. For (ultimately) decreasing $f$ the existence of local time is clearly not needed, whereas we have no conjecture for general $f$.
\end{rem}

\section*{Proof of Theorem \ref{theorem}}

	Before going into the proof let us fix some notation and facts needed below. For more definitions and background we refer for instance to \cite{bertoin} or \cite{Kyp}. The law of $\xi$ issued from $x\in \R$ will be denoted by $\P^x$, abbreviating $\P=\P^0$, and the characteristic exponent is defined as
\begin{align*}
	\Psi(\lambda):=-\log \E\big[\exp(i\lambda \xi_1)\big],\quad \lambda\in \R.
\end{align*}
We recall from Theorem V.1 of \cite{bertoin} that $\xi$ has local times $\big(L_t(x)\big)_{t\geq 0,x\in\R}$ if and only if 
\begin{align}\label{LT}
	\int_{-\infty}^\infty \mathcal{R}\left(\frac{1}{1+\Psi(r)}\right)\,dr<\infty.
\end{align}
This means that for any bounded measureable function $f:\R\to [0,\infty)$ the occupation time formula
\begin{align*}
	\int_0^t f(\xi_s)\,ds=\int_\R f(x) L_t(x)\,dx,\quad t\geq 0,
\end{align*}
holds almost surely.\\
A consequence of \eqref{LT} is also that points are non-polar. More precisely, a Theorem of Kesten and Bretagnolle states that $\P(\tau_x<\infty)>0$ for all $x>0$ if $\tau_x=\inf\{t: \xi_t=x\}$, see for instance Theorem 7.12 of \cite{Kyp}.

Throughout we assume $\xi$ is transient so that $\int_{-\eps}^\eps \mathcal {R}(\frac{1}{\psi(r)})\,dr<\infty$ and, consequently, \eqref{LT} implies $\int_{-\infty}^\infty \mathcal{R}\left(\frac{1}{\Psi(r)}\right)\,dr<\infty.$
But then Theorem II.16 of \cite{bertoin} implies that the potential measure
\begin{align*}
	U(dx)=\int_0^\infty \P\big(\xi_s\in dx\big)\,ds
\end{align*}
has a bounded density $u(x)$ with respect to the Lebesgue measure.\medskip

  We start with the easy direction of Theorem 1:
\begin{proof}[Proof of Theorem \ref{theorem}, Sufficiency of Integral Test] 
	Suppose that $\int_\R f(x)dx<\infty$. Since we assume that $\xi$ is transient and has a local time we can use the existence and boundedness of the potential density to obtain
\begin{align*}
	\E\Big[\int_0^\infty f(\xi_s)\,ds\Big]
=\int_\R f(x) \int_0^\infty \P(\xi_s\in dx)\,ds
=\int_\R f(x) u(x)\,dx
\leq \sup_{x\in \R}u(x) \int_\R f(x)\,dx<\infty.
\end{align*}
Since finiteness of the expectation implies almost sure finiteness the sufficiency of the integral test for almost sure finiteness of the perpetual integral is proved.
\end{proof}

%To prove the inverse direction first note the following simple fact:
%\begin{lemma}
%	Without loss of generality $f$ can be assumed to have compact support in $[0,\infty)$.
%\end{lemma}
%\begin{proof}
%	The claim follows from the local integrability of $f$ and the assumption that $\xi$ drifts to $+\infty$ and as such only spends a finite time in compact sets almost surely.
%\end{proof}

%For the rest of the proof we assume that $f$ has compact support in $[0,\infty)$.\\

For the reverse direction we use Jeulin's lemma, here is a simple version:
\begin{lemma}\label{Jeulin}
	Suppose $(X_x)_{x\in\R}$ are non-negative, non-trivial and identically distributed random variables on some probability space $(\Omega, \mathcal F, P)$, then
	\begin{align*}
		P\Big(\int_\R f(x) X_x\,dx<\infty
		\Big)=1\qquad \Longrightarrow\qquad \int_\R f(x)\,dx<\infty.
\end{align*}
\end{lemma}
\begin{proof}
	Since $X_x$ are identically distributed, we may choose $\eps>0$ so that $P(X_x>\eps)=\delta>0$ for all $x\in \R$. Since $\int_\R f(x) X_x\,dx$ is almost surely finite, there is some $N\in \N$ so that $P(A_N)> 1-\delta/2$ with $A_N=\{\int_\R f(x) X_x\,dx>N\}$. Hence, we have 
	\begin{align*}
		E\big[X_x 1_{A_N}\big]
		\geq \eps P(\{X_x>\eps\}\cap A_N)
		>\eps \delta/2>0.
	\end{align*}
	But then
	\begin{align*}
		N\geq N P(A_N)\geq E\Big[\int_\R f(x)X_x\,dx\,1_{A_N}\Big]=\int_\R f(x)E[X_x1_{A_N}]\,dx\geq \eps \delta/2 \int_\R f(x)\,dx.
	\end{align*}
	The proof is now complete.
\end{proof}

Note that there are different versions of Jeulin's lemma (see for instance \cite{Yano}). Most commonly, one refers to Jeulin's lemma if $P(\int_\R f(x) X_x\,dx<\infty)>0$ but more assumptions on the $X_x$ are posed. Those extra assumptions on $X_x$ are not satisfied in our setting but we can employ a 0-1 law that allows us to work with $P(\int_\R f(x) X_x\,dx<\infty)=1$.

\medskip
We would like to apply Jeulin's lemma via the occupation time formula
\begin{align*}
	\int_0^\infty f(\xi_s)\,ds=\lim_{t\uparrow\infty}\int_0^t f(\xi_s)\,ds=\lim_{t\uparrow \infty} \int_0^\infty f(x)L_t(x)\,dx=\int_0^\infty f(x) L_\infty(x)\,dx
\end{align*}
with $L_\infty(x):=\lim_{t\uparrow \infty} L_t(x)$. The argument is too simplistic because the distribution of $L_\infty(x)$ depends on $x$. Only if $\xi$ is spectrally negative the laws $L_\infty(x)$ are independent of $x$ by the strong Markov property. To make this idea work we work with randomized initial conditions instead. In what follows we chose a particularly convenient initial distribution motivated by a result from fluctuation theory (see Lemma 3 of \cite{BertoinSavov}). Our assumption $\E[\xi_1]<\infty$ implies that
\begin{align}\label{conv}
\P\big(\xi_{T_z}-z\in dy\big)\stackrel{z\to\infty}{\Longrightarrow}\rho(dy),
\end{align}
where $T_z=\inf\{t\geq 0:  \xi_t\geq z\}$ and $\rho$ is a non-degenerate probability law, called the stationary overshoot distribution. The convergence in \eqref{conv} is  a consequence of the quintuple law of \cite{Kyp} and the distribution $\rho$ can be written down explicitly.

\medskip
Since $\rho$ is the stationary overshoot distribution we have
\begin{align*}
	 \P^{\rho}\big( \xi_{T_a}-a\in dy\big):=\int  \P^{x}\big( \xi_{T_a}-a\in dy\big)\rho(dx)=\rho(dy),\quad \forall a>0,
\end{align*}
so that spatial stationarity holds due to the strong Markov property: under $\P^{\rho}$
\begin{align}\label{spatial}
	( \xi^{(a)}_t)_{t\geq 0}:=( \xi_{T_a+t}-a)_{t\geq 0}
\end{align}
has law $\P^{\rho}$ for all $a>0$. The stationarity property \eqref{spatial} will be the key to apply Jeulin's lemma.
\begin{lemma}\label{le2}
	For any $x>0$ we have 
	\begin{align*}
		 \P^{\rho}( L_\infty(x)\in dy)= \P^{\rho}( L_\infty(1)\in dy).
\end{align*}

\end{lemma}
\begin{proof}
First note that $\P^{\rho}(T_x<\infty)=1$ for all $x>0$ and ${L}_\cdot (x)$ only starts to increase at some time at or after  $T_x$. Then
	\begin{eqnarray*}
	\P^{\rho}(L_\infty(x)\in dy)
	&=&\int_0^\infty \P^{z}( L_\infty(x)\in dy)\P^\rho(\xi_{T_x}\in dz)\\
	&=&\int_0^\infty \P^{z+x}( L_\infty(x)\in dy)\P^\rho(\xi_{T_x}-x\in dz)\\
	&=& \int_0^\infty\P^{z+x}(L_\infty(x) \in dy)\rho(dz)\\
	&=& \int_0^\infty \P^{z}(L_\infty(0) \in dy)\rho(dz),
	\end{eqnarray*}
using the strong Markov property, the spatial stationarity of $\P^{\rho}$ and spatial homogeneity of $\xi$. Since the righthand side is independent of $x$ the proof is complete.
\end{proof}

Next, we use the Hewitt-Savage 0-1 law (see \cite{HS}) in order to get the weak version of Jeulin's lemma going. If $X^0,X^1,...$ denotes a sequence of random variables taking values in some measurable space, then an event $A\in \sigma(X^0,X^1,...)$ is called exchangeable if it is invariant under finite permutations (i.e. only finitely many indices are changed) of the sequence $X^0,X^1,...$. The Hewitt-Savage 0-1 law states that any exchangeable event of an iid sequence has probability $0$ or $1$.

\begin{lemma}\label{7}
	$\P\big(\int_0^\infty f(\xi_s)\,ds<\infty\big)\in \{0,1\}$.
\end{lemma}
\begin{proof}
	The idea of the proof is to write $\Lambda:=\{\int_0^\infty f(\xi_s)ds<\infty\}$ as an exchangeable event with respect to the iid increments of $\xi$ on intervals $[n,n+1]$ so that $\P(\Lambda)\in \{0,1\}$. Let $\mathcal D$ denote the RCLL functions $w:[0,1]\to \R$. If $\xi$ is the given L\'evy process, then define the increment processes as
	\begin{align*}
		(\xi^{n}_t)_{t\in [0,1]}=(\xi_{n+t}-\xi_n)_{t\in [0,1]}.
	\end{align*}	
	The L\'evy property implies that the sequence $\xi^0,\xi^1,...$ is iid on $\mathcal D$. Furthermore, note that $\xi$ can be reconstructed from the $\xi^n$ through
	\begin{align*}
		\xi_r=\xi^n_{r-n}+\sum_{i=0}^{n-1} \xi^i_1\qquad \forall r\in [n,n+1).
	\end{align*}
Using that $g_1:(w_t)_{t\in [0,1)} \mapsto (w_1)_{t\in [0,1)}$,  $g_2:(w,w')_{t\in [0,1)}\mapsto (w_t+w'_t)_{t\in [0,1)}$ and $g_3:(w_t)_{t\in [0,1)}\mapsto \int_0^1 f(w_s)ds$  are measurable mappings, there are measurable mappings $g^n:\mathcal D^n\to \R$ such that
\begin{align*}
\int_0^1 f\Big(\xi^n_r+\sum_{i=0}^{n-1} \xi_1^i\Big)\,dr=g^n(\xi^0,...,\xi^n).
\end{align*}
As a consequence we find that
\begin{align*}
	\Big\{\int_0^\infty f(\xi_s)\,ds<\infty\Big\}
	&=\Big\{\sum_{n=0}^\infty \int_n^{n+1}f(\xi_s)\,ds<\infty\Big\}\\
	&=\Big\{\sum_{n=0}^\infty \int_0^{1}f\Big(\xi^n_r+\sum_{i=0}^{n-1}\xi_1^i\Big)\,dr<\infty\Big\}\\
	&=\Big\{\sum_{n=0}^\infty g^n(\xi^0,...,\xi^n)<\infty\Big\}\\
	&\in \sigma(\xi^0,\xi^1,...).
\end{align*}
Since clearly $\Lambda$ is exchangeable for $\xi^0,\xi^1,...$ the Hewitt-Savage 0-1 law implies the claim.
	\end{proof}
\begin{lemma}\label{9}
	Suppose $\P(\int_0^\infty f(\xi_s)ds<\infty)=1$, then $\P^\rho(\int_0^\infty f(\xi_s)ds<\infty)=1$.
\end{lemma}
\begin{proof}
	The statement is obvious if $\xi$ is a subordinator, so we assume it is not.
	\smallskip
	
	Next we show that  $\P^x(\int_0^\infty f(\xi_s)ds<\infty)=1$ for any $x>0$. To see this we use the strong Markov property at $\tau_0=\inf\{t: \xi_t=0\}$ which is finite with positive probabillity since points in $\R$ are non-polar:
	\begin{align*}
		\P^x\Big(\int_0^\infty f(\xi_s)\,ds<\infty\Big)
		&\geq \P^x\Big(\int_{\tau_0}^\infty f(\xi_s)\,ds<\infty, \tau_0<\infty\Big)\\
		&=\P^x\Big(\int_{0}^\infty f(\xi_{s+\tau_0}-\xi_{\tau_0})\,ds<\infty, \tau_0<\infty\Big)\\
		&=\P^0\Big(\int_0^\infty f(\xi_s)\,ds<\infty\Big)\P^x(\tau_0<\infty)\\
		&>0.
	\end{align*}	
	But then the 0-1 law of Lemma \ref{7} implies that $\P^x\Big(\int_0^\infty f(\xi_s)\,ds<\infty\Big)=1$. Finally, we obtain	
	\begin{align*}
		\P^\rho\Big(\int_0^\infty f(\xi_s)\,ds<\infty\Big)&=\int_\R \P^x\Big(\int_0^\infty f(\xi_s)\,ds<\infty\Big)\rho(dx)
		=\int_\R \rho(dx)
		=1
	\end{align*}
	and the proof is complete.
\end{proof}

Now we are ready to prove the more delicate part of Theorem \ref{theorem}.
\begin{proof}[Proof of Theorem \ref{theorem}, Necessity of Integral Test]
Suppose $\P(\int_0^\infty f(\xi_s)\,ds<\infty)>0$ which then implies $\P(\int_0^\infty f(\xi_s)\,ds<\infty)=1$ by Lemma \ref{7}. Hence,  $\P^{\rho}(\int_0^\infty f(\xi_s)\,ds<\infty)=1$ by Lemma \ref{9}. Using the occupation time formula we get
\begin{align*}
	\int_0^\infty f( \xi_s)\,ds
	=\lim_{t\to \infty} \int_0^t f( \xi_s)\,ds
	=\lim_{t\to \infty} \int_\R f(x)  L_t(x)\,dx
	= \int_\R f(x) L_\infty(x)\,dx\qquad \P^{\rho}\text{-a.s.}
\end{align*}	
In Lemma \ref{le2} we proved that $ L_\infty(x)$ is independent of $x$ under $\P^{\rho}$ so that Jeulin's Lemma implies $\int_\R f(x)dx<\infty$.
\end{proof}

\subsection*{Acknowledgement}
The authors thank Jean Bertoin for several discussions on the topic.


\begin{thebibliography}{99}
\bibitem{bertoin} J. Bertoin: "L\'evy processes." Cambridge Tracts in Mathematics 121, Cambridge University Press, 1996.

\bibitem{BertoinSavov} J. Bertoin, M. Savov: "Some applications of duality for L\'evy processes in a half-line." Bull. London Math. Soc. 43, pp. 97-110, 2011.
\bibitem{Duf} D. Dufresne: "The distribution of a perpetuity, with applications to risk theory and pension funding." Scand. Actuarial J., pp. 39-79, 1990.
\bibitem{HS} E. Hewitt, L. J. Savage: ``Symmetric measures on Cartesian products.'' Trans. Amer. Math. Soc., 80, pp. 470-501 (1955).
\bibitem{Maller} K.B. Erickson, R.A. Maller: "Generalised Ornstein-Uhlenbeck processes and the convergence of L\'evy integrals." S\'eminaire de Probabilit\'es XXXVIII, 2005.

\bibitem{Yor2} D. Khoshnevisan, P. Salminen, M. Yor: "A note on a.s. finiteness of perpetual integral functionals of diffusions." Electr. Comm. Probab., Vol. 11, pp. 108-117.
\bibitem{Kyp} A. Kyprianou, J.C. Pardo, V. Rivero: "Exact and asymptotic n-tuple laws at first and last passage." Ann. Appl. Probab. Volume 20, pp. 522-564, 2010.

\bibitem{Kyp} A. Kyprianou: ``Fluctuations of L\'evy Processes with Applications.'' Springer (2013).

\bibitem{Schilling} R. Schilling, Z. Vondra\v{c}ek:  ``Absolute continuity and singularity of probability measures induced by a purely discontinuous Girsanov transform of a stable process.'' arXiv:1403.7364


\bibitem{Yano}
A. Matsumoto, K. Yano: "On a zero-one law for the norm process of transient random walk.'' S\'eminaire de Probabilit\'es XLIII, pp. 105-126, 2011.

\bibitem{Yor} P. Salminen, M. Yor: "Properties of perpetual integral functionals of Brownian motion with drift." Ann. I.H.P., 41, pp. 335-347, 2005.
\bibitem{Yor2} P. Salminen, M. Yor: "Perpetual integral fuctionals as hitting and occupation times." Electr. J. Probab., Vol. 10, p. 371-419, 2005. 


\end{thebibliography}
\end{document}